\documentclass[12pt,leqno,twoside]{amsart}
\usepackage[noadjust]{cite}
\usepackage{indentfirst}
\usepackage{amssymb,amsmath,amsthm,soul,color}
\usepackage{t1enc}
\usepackage[cp1250]{inputenc}
\usepackage{a4,indentfirst,latexsym}
\usepackage{cite,enumitem,graphicx}
\usepackage[colorlinks=true,urlcolor=blue,
citecolor=red,linkcolor=blue,linktocpage,pdfpagelabels,
bookmarksnumbered,bookmarksopen]{hyperref}
\usepackage[english]{babel}
\usepackage[left=2.61cm,right=2.61cm,top=2.72cm,bottom=2.72cm]{geometry}
\usepackage[normalem]{ulem}

\usepackage[colorinlistoftodos]{todonotes}

\newtheorem{Th}{Theorem}[section]
\newtheorem{Prop}[Th]{Proposition}
\newtheorem{Lem}[Th]{Lemma}
\newtheorem{Cor}[Th]{Corollary}
\newtheorem{Def}[Th]{Definition}

\newcommand{\wt}{\widetilde}

\newcommand{\vp}{\varphi}

\newcommand{\eps}{\varepsilon}

\def\div{\mathop{\mathrm{div}\,}}

\def\span{\mathrm{span}}

\def\Z{\mathbb{Z}}

\def\R{\mathbb{R}}

\def\J{\mathcal{J}}

\def\D{\mathcal{D}}

\def\irn{\int_{\rn}}
\def\rn{\mathbb{R}^N}

\newcommand{\cC}{{\mathcal C}}
\newcommand{\cD}{{\mathcal D}}
\newcommand{\cE}{{\mathcal E}}
\newcommand{\cF}{{\mathcal F}}

\newcommand{\cH}{{\mathcal H}}
\newcommand{\cI}{{\mathcal I}}
\newcommand{\cJ}{{\mathcal J}}

\newcommand{\cN}{{\mathcal N}}
\newcommand{\cO}{{\mathcal O}}

\newcommand{\cS}{{\mathcal S}}

\newcommand{\cX}{{\mathcal X}}
\newcommand{\cY}{{\mathcal Y}}
\newcommand{\cZ}{{\mathcal Z}}

\newcommand{\DF}{\cD_{\cF}}
\newcommand{\UU}{\mathbf{U}}
\newcommand{\VV}{\mathbf{V}}
\newcommand{\BB}{\mathbf{B}}

\newcommand{\mbS}{\mathbb{S}}
\newcommand{\SO}{\cS\cO}

\newcommand{\vphi}{\varphi}

\newcommand{\De}{\Delta}

\newcommand{\Om}{\Omega}

\def\curlop{\nabla\times}

\numberwithin{equation}{section}

\title[Cylindrically symmetric equations]{Multiple solutions to cylindrically symmetric curl-curl problems and related Schr\"odinger equations with singular potentials}

\author[M. Gaczkowski]{Michał Gaczkowski}
\author[J. Mederski]{Jaros\l aw Mederski}
\author[J. Schino]{Jacopo Schino}

\address[J. Mederski]{\newline\indent
	Institute of Mathematics,
	\newline\indent 
	Polish Academy of Sciences,
	\newline\indent 
	ul. \'Sniadeckich 8, 00-656 Warsaw, Poland
}
\email{\href{mailto:jmederski@impan.pl}{jmederski@impan.pl}}

\address[J. Schino]{\newline\indent
	Institute of Mathematics,
	\newline\indent 
	Polish Academy of Sciences,
	\newline\indent 
	ul. \'Sniadeckich 8, 00-656, Warsaw, Poland
	\newline\indent
	and
	\newline\indent
	Department of Mathematics,
	\newline\indent
	North Carolina State University,
	\newline\indent
	2311 Stinson Drive, 27607 Raleigh, NC, USA
}
\email{\href{mailto:jschino@ncsu.edu}{jschino@ncsu.edu}}

\address[M. Gaczkowski]{\newline\indent 
	Faculty of Mathematics and and Information Science,
	\newline\indent 
	Warsaw University of Technology,
	\newline\indent 
	ul. Koszykowa 75, 00-662, Warsaw, Poland
}
\email{\href{mailto:m.gaczkowski@mini.pw.edu.pl}{m.gaczkowski@mini.pw.edu.pl}}

\dedicatory{In memory of Enrico Jannelli}

\date{\today}

\begin{document}
	

\begin{abstract}
	We look for multiple solutions $\UU\colon\mathbb{R}^3\to\mathbb{R}^3$ to the curl-curl problem
	\[
	\curlop\curlop\UU=h(x,\UU),\qquad x\in\R^3,
	\]
	with a nonlinear function $h\colon\R^3\times\R^3\to\R^3$ which is critical in $\R^3$, i.e., $h(x,\UU)=|\UU|^4\UU$, or has subcritical growth at infinity. If $h$ is radial in $\UU$ and $a=1$ below, then we show that the solutions to the problem above are in one to one correspondence with the solutions to the following Schr\"odinger equation
	\[
	-\Delta u+\frac{a}{r^2}u=f(x,u),\qquad u\colon\mathbb{R}^3\to \mathbb{R},
	\]
	where $x=(y,z)\in \mathbb{R}^2\times \mathbb{R}$, $r=|y|$, and $a\geq 0$.
	In the critical case, the multiplicity problem for the latter equation has been studied only in the autonomous case $a=0$ and the available methods seem to be insufficient for the problem involving the singular potential, i.e., $a\neq 0$, due to the lack of conformal invariance. Therefore we develop methods for the critical curl-curl problem and show  the multiplicity of bound states for both equations. In the subcritical case, instead, studying the Schr\"odinger equation in higher dimensions, we find infinitely many bound states for both problems.
\end{abstract}

\maketitle
	
	\section*{Introduction}
	\setcounter{section}{1}
	
	We look for solutions $\UU\colon\R^3\to\R^3$ to the following {\em curl-curl problem}
	\begin{equation}\label{eqV}
		\curlop\curlop\UU=h(x,\UU)\qquad\text{in }\R^3
	\end{equation}
	with a nonlinear function $h\colon\R^3\times\R^3\to\R^3$,
	which arises in the study of the  propagation of time-harmonic electromagnetic fields in a nonlinear medium by means of Maxwell's equations and the constitutive laws, see \cite{Stuart:1991,Stuart:1993,Agrawal,Doerfler,Mederski} and the references therein.
	The first existence results concerning \eqref{eqV} have been obtained  for the cylindrically symmetric media in \cite{AzzBenDApFor} and for nonsymmetric ones in \cite{Mederski} by means of variational methods for subcritical $h$ at infinity.
		
	Most of known results (see the survey \cite{MeScSurvey} and the references therein) concern the subcritical regime, and the problem in the purely critical case $h(x,\UU)=|\UU|^4\UU$ was an open question for some time. Recently, the second author and Szulkin \cite{MeSzu} obtained finally the existence of a least energy solution  to \eqref{eqV} in the critical case by developing new concentration compactness argument and variational approach for the curl-curl operator. However, the problem of the existence of multiple solutions to the critical curl-curl problem remained open and will be investigated below as one of the main aims of the work.
 
	 On the other hand, we are also interested in finding solutions $u\colon\R^N\to\R$ to the problem
	\begin{equation}\label{eqS}
		-\Delta u+\frac{a}{r^2}u= f(x,u) \qquad\textnormal{in }\R^N,
	\end{equation}
	where $x=(y,z)\in\R^K\times\R^{N-K}$, $N>K\ge 2$, $r=|y|$ is the Euclidean norm in $\R^K$, $a>a_0\in(-\infty,0]$.
	Here $f\colon\R^N\times\R\to\R$ is a nonlinear function with critical growth, i.e., $f(x,u)=|u|^{2^*-2}u$, or subcritical growth at infinity, see assumptions (F1)--(F5) below. The problem appears in the study of stationary solutions to nonlinear Schr\"odinger or Klein--Gordon equations \cite{BadBenRol}.
	
	We also study the relation between \eqref{eqV} and \eqref{eqS} as follows. Suppose that $N=3$, $K=2$, $a=1$, and $h(x,\alpha w)=f(x,\alpha)w$ for $\alpha\in\R$, $w\in \R^3$ such that $|w|=1$, and $x\in\R^3$.
	Then one can easily calculate that if $u(x)=u(r,x_3)$ with $r=|(x_1,x_2)|$ is a classical solution to \eqref{eqS}, then 
	\begin{equation}\label{eq:formulauU}
	\UU(x)=\frac{u(x)}{\sqrt{x_1^2+x_2^2}}\Bigl(\begin{smallmatrix}
	-x_2\\
	x_1\\
	0
	\end{smallmatrix}\Bigr),\quad x=(x_1,x_2,x_3)\in\R^3\setminus(\{0\}\times\{0\}\times\R),
	\end{equation}
	satisfies $\div(\UU)=0$ and $\curlop(\curlop \UU)=-\Delta\UU=h(x,\UU)$ for $x\in\R^3\setminus(\{0\}\times\{0\}\times\R)$, cf. \cite{MederskiJFA,Zeng}. Our first aim is to show that \eqref{eq:formulauU} is the key to build solutions to \eqref{eqS} from solutions to \eqref{eqV} and vice versa also in the context of nonclassical solutions, which is not immediate and will be demonstrated in Theorem \ref{ScalVec} below. In particular, it will be crucial in the Sobolev-critical regime to study \eqref{eqS} by means of \eqref{eqV}. As an additional application of Theorem \ref{ScalVec}, we will prove existence results for \eqref{eqS} in the Sobolev-noncritical regime, thus providing new results also for \eqref{eqV}.

By a {\em solution} to \eqref{eqV} we mean a critical point $\UU\in\cD^{1,2}(\R^3,\R^3)$ of $\cE$, where the functional $\cE\colon\cD^{1,2}(\R^3,\R^3)\to\R$ is defined as
\begin{equation}\label{eq:I}
	\cE(\UU)=\frac12\int_{\R^3}|\curlop\UU|^2\,dx-\int_{\R^3}H(x,\UU)\,dx,
\end{equation}
$H(x,\UU):=\int_{0}^{1}\langle h(x,t\UU),\UU\rangle\,dt$, and $\cD^{1,2}(\R^3,\R^3)$ is the completion of $\cC_0^{\infty}(\R^3,\R^3)$ with respect to the norm $|\nabla \UU|_2$ (observe that $\cE$ is of class $\cC^1$ under reasonable assumptions about $h$). Here and in the sequel, $|\cdot|_q$ denotes the $L^q$-norm for $q\in[1,\infty]$. Clearly every solution to \eqref{eqV} is a {\em weak solution} in $\R^3$, i.e., $\cE'(\UU)(\phi)=0$ for every $\phi\in\cC_0^{\infty}(\R^3,\R^3)$. Inspired by \cite{Stuart:1991,AzzBenDApFor}, we define $\DF$ as the subspace of $\cD^{1,2}(\R^3,\R^3)$ consisting of vector fields of the form \eqref{eq:formulauU}, which will be defined rigorously in Section \ref{sec:vector}. 

Observe that the kernel of $\curlop(\cdot)$ is of infinite dimension since $\curlop(\nabla \phi)=0$ for all $\phi$ of class $\cC^2$ and $\cE$ is strongly indefinite, i.e., it is unbounded from above and from below, even on subspaces of finite codimension. In order to avoid the indefiniteness we use Theorem \ref{ScalVec}; in particular, the property that $\curlop \curlop (\cdot) = -\De$ in $\DF$ (Sobolev-critical regime) or the one that $\UU\in\D_\cF$ is a solution to \eqref{eqV} if and only if $u\in X_{\cO}$ is a solution to \eqref{eqS} with $a=1$, $N=3$, and $K=2$ (Sobolev-noncritical regime), where $X_\cO$ is defined below.

Regarding \eqref{eqS}, we introduce the energy functional
\begin{equation*}
\cJ(u) = \frac12 \int_{\R^N} |\nabla u|^2 + \frac{a}{r^2} |u|^2 \, dx - \int_{\R^N} F(x,u) \, dx,
\end{equation*}
which, under usual assumptions about $f$, is of class $\cC^1$ in  $X \subset \D^{1,2}(\R^N)$, where $F(x,u) := \int_0^u f(x,t) \, dt$, $\D^{1,2}(\R^N)$ is the completion of $\cC_0^{\infty}(\R^N)$ with respect to the usual norm $|\nabla u|_2$, and 
$$
X:=\Big\{u\in\cD^{1,2}(\R^N):\int_{\R^N}\frac{|u|^2}{r^2}\,dx<\infty\Big\},
$$
endowed with the norm $\|u\|=(|\nabla u|_2^2+|u/r|_2^2)^{1/2}$. Due to the singular term, we consider the group $\cO := \cO(K) \times \{I_{N-K}\}\subset \cO(N)$, which acts isometrically on $X$, and let $X_{\cO}$ denote the subspace of $X$ consisting of invariant functions with respect to $\cO$. Here and in the sequel, $\cO(d)$ is the orthogonal group in $\R^d$. Note that this is equivalent to requiring that the functions be invariant with respect to $\SO(K)\times\{I_{N-K}\}$, since for every $x,y\in\mbS^{d-1}$, $d\geq 2$, there exists $g\in\SO(d)$ such that $gy=x$, where $\SO(K) \subset \cO(K)$ stands for the special orthogonal group in $\R^K$.

If $K>2$, then $$\irn\frac{|u|^2}{r^2}\,dx\le\biggl(\frac{2}{K-2}\biggr)^2\irn|\nabla u|^2\,dx$$ for every $u\in\cD^{1,2}(\rn)$ (see \cite{BadTar}) and $X$ and $\cD^{1,2}(\R^N)$ coincide. Note also that
$$\cJ'(u)(\phi)=\int_{\R^N}\langle \nabla u,\nabla \phi\rangle+\frac{a}{r^2}u\phi\,dx - \int_{\R^N} f(x,u)\phi\,dx$$
need not be finite for $K=2$, $u\in X$, $\phi\in\cC_0^{\infty}(\R^N)$. Therefore, by
a {\em solution}  to \eqref{eqS} we mean a critical point $u\in X$ of $\cJ$.
By \textit{ground state solution} to \eqref{eqS} in $X_{\cO}$ we mean a nontrivial solution to \eqref{eqS} that minimizes $\{\cJ(u):u\in X_{\cO}\setminus\{0\}$, $u$ is a solution to \eqref{eqS}$\}$. If $K>2$, then $\cC_0^{\infty}(\R^N)$ is dense in $X$ and, by the Palais principle of symmetric criticality \cite{Palais}, critical points of $\cJ|_{X_{\cO}}$ correspond to critical points of $\cJ$ and are {\em weak solutions in $\R^N$}, i.e., $\cJ'(u)(\phi)=0$ for every $\phi\in\cC_0^{\infty}(\R^N)$. If $K=2$, then in view of \cite[Proposition 5]{BadBenRol} any nonnegative solution  to \eqref{eqS} is a weak solution in $\R^N$.

We begin by investigating the problems with the critical nonlinearities when $N=3$:
$$
h(x,\UU) = |\UU|^{4} \UU \hbox{ and } f(x,u)=|u|^{4}u \quad \hbox{ for } u \in \R, \UU \in \R^3.
$$
The existence of a solution to \eqref{eqV} that minimizes $\cE$ among \textit{all} the nontrivial solutions (not only those of the form \eqref{eq:formulauU}) has been recently obtained in \cite{MeSzu}. In order to  investigate  multiple solutions, inspired by \cite{Ding,ClPis}, we introduce the following $\SO(2)\times\SO(2)$-group action on $\cD^{1,2}(\R^3,\R^3)$.

\begin{Def}\label{def:sym}
	For $g_1,g_2\in\SO(2)$ we denote $g=\bigl(\begin{smallmatrix}
		g_1 & 0\\
		0 & g_2
	\end{smallmatrix}\bigr)\equiv(g_1,g_2)\in\SO(2)\times\SO(2)$. We say that $\UU\in\cD^{1,2}(\R^3,\R^3)$ is $\SO(2)\times\SO(2)$-symmetric if and only if
	\[
	\frac{\UU\Bigl(\pi\bigl(g\pi^{-1}(x)\bigr)\Bigr)}{\vphi\Bigl(\pi\bigl(g\pi^{-1}(x)\bigr)\Bigr)}=\frac{\widetilde{g_1}\UU(x)}{\vphi(x)}
	\]
	for every $g_1,g_2\in\SO(2)$ and a.e. $x\in\R^3$, where $\vp(x)=\sqrt{\frac{2}{1+|x|^2}}$, $\pi\colon\mathbb{S}^3\to\R^3\cup\{\infty\}$ is the stereographic projection of the unit sphere $\mathbb{S}^3 \subset \mathbb{R}^4$, and $\wt g_1=\bigl(\begin{smallmatrix}
		g_1 & 0\\
		0 & 1
	\end{smallmatrix}\bigr)$.
\end{Def}
The subspace of $\SO(2)\times\SO(2)$-symmetric vector fields is denoted by
$\cD_{\SO(2)\times\SO(2)}$ and
the main result concerning \eqref{eqV} in the critical case reads as follows.

\begin{Th}\label{critV}
	There exists a sequence $(\UU_n)\subset\cD_{\SO(2)\times\SO(2)}$ of solutions to
	\begin{equation}\label{eq:crV}
		\curlop\curlop\UU=|\UU|^{4}\UU\qquad\text{in }\R^3
	\end{equation}
	such that $\cE(\UU_n)\to\infty$ as $n\to\infty$. Each $\UU_n$ is of the form \eqref{eq:formulauU}. 
\end{Th}

Since each energy level $\cE(\UU_n)$ has a minimax characterization, $\UU_n$ can be considered as a bound state solution. In order to overcome the difficulties owing to the indefiniteness mentioned above, we consider vector fields $\UU$ of the form \eqref{eq:formulauU} and we exploit the property that such vector fields are divergence-free in the distributional sense (cf. Lemma \ref{DivFree}). Similarly as in \cite{AzzBenDApFor}, this allows us to reduce the curl-curl operator to the vector Laplacian by means of the appropriate group action and the Palais principle of symmetric criticality, see Section \ref{sec:vector} for details. Of course, we need to work with vector fields of the form \eqref{eq:formulauU} that are additionally $\SO(2)\times\SO(2)$-symmetric (cf. Lemma \ref{lem:inf}). We show that such symmetric vector fields are compactly embedded into $L^6(\R^3,\R^3)$. We would like to emphasize that using simply the $\Gamma$-invariance from \cite{ClPis,Ding} (recalled explicitly  below) for each component of $\UU$ does not allow us to reduce the problem to divergence-free vector fields and it is not clear whether we can obtain solutions in such a setting. However, our equivariant symmetry introduced in Definition \ref{def:sym} allows us to finally find infinitely many solutions to \eqref{eq:crV}.

Turning to \eqref{eqS}, we point out that in \cite{BadGuiRol}, Badiale, Guida, and Rolando found a ground state solution in $X_{\cO}$ for $a>0$. 
An immediate consequence of Theorem \ref{ScalVec} is the existence of a solution to \eqref{eqV} of the form \eqref{eq:formulauU} that minimizes the energy functional $\cE$ among all the nontrivial solutions of the same form, although it is not clear whether it is a least-energy solution in the general sense specified above. Multiplicity results for \eqref{eqS} in the critical case, nonetheless, are not known up to our knowledge unless $a=0$.
Observe that, if $a=0$, then the most prominent result is due to Ding \cite{Ding}, which establishes the existence of infinitely many sign-changing solutions $(u_n)$ invariant under the conformal action of $\Gamma:=\cO(2)\times\cO(2)$ on $\R^3$, shortly {\em $\Gamma$-invariant}, induced by the stereographic projection, i.e., for each $\gamma\in\Gamma$ we set $\wt\gamma(x):=(\pi\circ\gamma^{-1}\circ\pi^{-1})(x)$ and $\gamma u(x):=|\hbox{det} \wt\gamma'(x)|^{1/6}u(\wt \gamma(x))$ for all $\gamma\in\Gamma$ and a.e. $x\in\R^N$. Such a group action restores compactness, e.g. similarly as in Clapp and Pistoia \cite{ClPis}, i.e., the subspace of $\Gamma$-invariant functions in $\cD^{1,2}(\R^3)$ is compactly embedded into $L^{6}(\R^3)$. However, if $a\neq 0$, then one easily checks that $\cJ$ is not $\Gamma$-invariant due to the term $u/r^2$ and the lack of the conformal invariance, hence this approach no longer applies. Therefore, in order to solve \eqref{eqS} in $N=3$ and with $a=1$, we consider the infinitely many solutions to \eqref{eqV}, which does not involve a singular term, given by Theorem \ref{critV} involving the symmetry  from Definition \ref{def:sym} and use the relation \eqref{eq:formulauU} by means of Theorem \ref{ScalVec}, obtaining the following result.

\begin{Cor}\label{critS}
	There exists a sequence $(u_n)\subset X_{\cO}$ of solutions to
	\[
	-\Delta u+\frac{u}{r^2}=|u|^{4}u \qquad\text{in }\R^3
	\]
	such that $\cJ(u_n)\to\infty$ as $n\to\infty$. Moreover, each $|u_n|$ is $\Gamma$-invariant.
\end{Cor}
It remains an open question whether there are infinitely many solutions to \eqref{eqS} in the critical case with $a\neq 0,1$ or $(N,K) \ne (3,2)$; in particular, concerning the latter, the reason is that $\SO(K)$ is not abelian if $K\ge3$, a property our approach heavily relies on.

If $N>K>2$, then we would like to mention the existence of a ground state solution for $-\bigl(\frac{K-2}{2}\bigr)^2<a<0$ and the nonexistence of ground states solution for $a>0$, see \cite[Theorem 1.2]{ChabrSzW}.

Concerning the Sobolev-noncritical regime, in order to enhance the applications of Theorem \ref{ScalVec} we work directly with \eqref{eqS} and collect some assumptions about $f$.
\begin{itemize}
	\item [(F1)] $f\colon\R^N\times\R\to\R$ is a Carath\'eodory function (i.e., measurable in $x\in\R^N$ for every $u\in\R$, continuous in $u\in\R$ for a.e. $x\in\R^N$). We assume that $f$ is $\cO$-invariant with respect to $x$, i.e., $f(gx,u)=f(x,u)$ for $g\in\cO$, for a.e. $x\in\R^N$, and for every $u\in\R$. Moreover, $f$ is $\Z^{N-K}$-periodic in the last $N-K$ components of $x$, i.e., $f(x,u)=f\bigl(x+(0,z'),u\bigr)$ for every $u\in\R$, a.e. $x\in\R^N$, and a.e. $z'\in\Z^{N-K}$. 
	\item [(F2)] $\displaystyle\lim_{|u|\to 0}\frac{f(x,u)}{|u|^{2^*-1}}=\lim_{|u|\to\infty}\frac{f(x,u)}{|u|^{2^*-1}}=0$ uniformly with respect to $x\in\R^{N}$, where $2^*=\frac{2N}{N-2}$.
	\item [(F3)] $\displaystyle\lim_{|u|\to\infty}\frac{F(x,u)}{|u|^2}=\infty$ uniformly with respect to $x\in\R^{N}$.
	\item [(F4)] $\displaystyle u\mapsto\frac{f(x,u)}{|u|}$ is nondecreasing on $(-\infty,0)$ and on $(0,\infty)$ for a.e. $x\in\R^{N}$.
\end{itemize}

The first existence result in the Soblev-noncritical regime reads as follows.
	\begin{Th}\label{ExMul}
		Suppose that $a>-\bigl(\frac{K-2}{2}\bigr)^2$ and (F1)--(F4) hold. Then
		there exists a ground state solution $u$ to \eqref{eqS} in $X_{\cO}$. If, in addition, $f$ is odd in $u\in\R$, then $u$ is nonnegative and	
		 \eqref{eqS} has infinitely many geometrically $\Z^{N-K}$-distinct solutions in $X_{\cO}$.
	\end{Th}
	
	 Two solutions are called {\em geometrically $\Z^{N-K}$-distinct} if and only if one cannot be obtained via a translation of the other in the last $N-K$ variables by a vector in $\Z^{N-K}$.

	The growth conditions (F1)--(F4) are provided in \cite[Section 7]{MeScSz} for the problem \eqref{eqS} with $a=0$, $\Z^N$-periodic $f$, $\cO=\{I_N\}$ and Theorem \ref{ExMul} is  known in this particular case; see \cite[Theorem 7.1]{MeScSz}. These assumptions imply that $f(x,u)u\geq 2F(x,u)\geq 0$. However, if $F$ does not depend on $y$, we may consider also sign-changing nonlinearities under the following weaker variant of the Ambrosetti--Rabinowitz condition \cite{AmbRab}:
	\begin{itemize}
		\item[(F5)]  There exists $\gamma>2$ such that $f(z,u)u\geq\gamma F(z,u)$ for every $u\in\R$ and a.e. $z\in\R^{N-K}$ and there exists $u_0\in\R$ such that $\mathrm{essinf}_{z\in\R^{N-K}}F(z,u_0)>0$.
	\end{itemize}

	\begin{Th}\label{ExMul2}
		Suppose that $a>-\bigl(\frac{K-2}{2}\bigr)^2$, (F1)--(F2) and (F5) hold, and $f$ does not depend on $y$. Then
		there exists a nontrivial solution $u$ to \eqref{eqS} in $X_{\cO}$. 
	\end{Th}

	Notice that every solution $u$ to \eqref{eqS} can be supposed to be nonnegative if $f(z,u)\ge 0$ for every $u\le 0$ and a.e. $z\in\R^{N-K}$ because, in this case,
	\[
	0\ge-\|u_-\|^2=\int_{\R^N}\langle\nabla u,\nabla u_-\rangle+\frac{a}{r^2}uu_-\,dx=\int_{\R^N}f(z,u)u_-\,dx=\int_{\R^N}f(z,-u_-)u_-\,dx\ge 0
	\]
	(where $u_-:=\max\{-u,0\}$ denotes the negative part of $u$), therefore $u_-=0$ and $u\ge 0$.

	Recall that if $f(x,u)=f(u)$ does not depend on $x$ and $a=1$, then  Badiale, Benci, and Rolando \cite{BadBenRol} found a nontrivial and nonnegative solution to \eqref{eqS} under more restrictive assumptions than in Theorem \ref{ExMul2},  in particular (cf. assumption (f$_1$) there), they assumed the double-power-like behaviour $|f(u)|\leq C\min\{|u|^{p-1},|u|^{q-1}\}$ for $u\in\R$, some constant $C>0$, and $2<p<2^*<q$.  For instance, the result of \cite{BadBenRol} does not allow nonlinearities such as
	\begin{equation*}
	f(u):=\begin{cases}
	|u|^{p-2}u\ln(1+|u|) & \hbox{for } |u|\geq 1 \\ 
	\ln(2) \frac{|u|^{2^*-2}u}{1-\ln(|u|)} & \hbox{for } 0<|u|<1\\
	0 & \hbox{for } u=0,
	\end{cases}
	\end{equation*}
	where $2\leq p <2^*$, not even if considering $f \chi_{[0,\infty)}$, where $\chi_{[0,\infty)}$ is the characteristic function of $[0,\infty)$. Observe that $f(x,u)=\Gamma(x)f(u)$ satisfies (F1)--(F4), where 
	$\Gamma\in L^{\infty}(\R^N)$ is
	$\cO$-invariant, $\Z^{N-K}$-periodic in the last $N-K$ components, positive and bounded away from $0$. (F5) is satisfied if and only if $p>2$ and $\Gamma$ does not depend on $y$.
	
	In general, if $f$ satisfies (F1), (F2), (F5) and does not depend on $y$ as in Theorem \ref{ExMul2}, then 
	$f \chi_{[0,\infty)}$ satisfies the same assumptions, but clearly (F3) does not hold.  
	
	We show that the problem in Theorem \ref{ExMul2} has the mountain pass geometry; however, the presence of the singular potential and the nonlinearity of general type cause difficulties in the concentration-compactness analysis. In \cite{BadBenRol}, dealing with an autonomous double-power-like nonlinearity, the authors provided a technical analysis involving translations and rescaling in $\cD^{1,2}(\R^N)$ \cite[Section 4]{BadBenRol}. We show, however, that this rather involved argument can be replaced by a Lions-type lemma for functions in $X_{\cO}$, where we only make use of translations, see Lemma \ref{Lions} for the precise statement. In order to obtain a ground state solution in $X_\cO$ and infinitely many solutions to \eqref{eqS}, instead, we apply the critical point theory from \cite[Section 3]{MeScSz}.
	
In view of Theorems \ref{ExMul}, \ref{ExMul2} and taking into account Theorem \ref{ScalVec}, we obtain the existence of solutions of the form \eqref{eq:formulauU} to \eqref{eqV}. Curl-curl problems in a bounded domain or in $\R^3$ have been recently studied e.g. in \cite{AzzBenDApFor,BartschReich,HirschReichel,BartschMederski1,BartschMederski2,Mederski,MeScSz} under different hypotheses on $h$  assuming $h$ is subcritical, i.e., $h(x,\UU)/|\UU|^5\to 0$ as $|\UU|\to\infty$. A multiplicity result in $\R^3$ has been recently obtained in \cite{MeScSz}. Below we provide the multiplicity result inferred by Theorem \ref{ExMul2} under more general growth assumption, although we have to assume the radial symmetry of $h$ as follows.

\begin{Cor}\label{CorExMul}
	Suppose that $h(x,\alpha w)=f(x,\alpha)w$ for $\alpha\in\R$, $w\in \R^3$ such that $|w|=1$, and a.e. $x\in\R^3$.
	\begin{itemize}
		\item [(a)] If (F1)--(F4) hold, then \eqref{eqV} has infinitely many $\Z$-distinct solutions in $\DF$, one of which minimizes  $\{\cE(\VV):\VV\in\DF\setminus\{\mathbf{0}\}$, $\VV$ is a solution to \eqref{eqV}$\}$.
		\item [(b)] If (F1)--(F2), (F5) hold and $h$ does not depend on $y$, then \eqref{eqV} has a nontrivial solution $\UU\in\DF$.
	\end{itemize}
\end{Cor}

Note that the condition $h(x,\alpha w)=f(x,\alpha)w$ for $\alpha\in\R$ and $w\in \R^3$ such that $|w|=1$ means that $h$ is $\cO(3)$-equivariant (radial) with respect to $\UU$, i.e., $h(x,g \UU)=g h(x,\UU)$ for $g\in\cO(3)$, $\UU\in\R^3$, and a.e. $x\in\R^3$; however, in general, we cannot expect that solutions obtained in Corollary \ref{CorExMul} preserve this symmetry. Indeed,
it follows from \cite[Theorem 1.1]{BartschReich} that any $\cO(3)$-equivariant solution to \eqref{eqV} is trivial provided that $f(x,u)\neq 0$ for $u\neq 0$ and a.e. $x\in\R^3$.

Finally, observe that in Corollary \ref{CorExMul} we obtain weak solutions to \eqref{eqV} by critical points of $\cJ$ from Theorem \ref{ExMul}, although we do not know, in general, whether they are weak solutions to \eqref{eqS} in $\R^3$.

The paper is structured as follows. In Section \ref{sec:vector} we build the functional setting for \eqref{eqS} and \eqref{eqV} and prove that solutions to \eqref{eqS} in $X_{\cO}$ are in one-to-one correspondence to solutions to \eqref{eqV} in $\DF$.  In Section \ref{sec:crit} we study the critical problems in dimension $N=3$, while in Section \ref{sec:noncrit} we study the noncritical ones.

\section{An equivalence result}\label{sec:vector}

In this section we deal with the case $N=3$ and, consequently, $K=2$. In particular, $2^*=6$ and $r=r_x=\sqrt{x_1^2+x_2^2}$ for $x=(x_1,x_2,x_3)\in\R^3$. For $g\in\SO(2)$ we denote $\widetilde{g}=\bigl(\begin{smallmatrix}
g & 0\\
0 & 1
\end{smallmatrix}\bigr)\in\SO=\SO(2)\times \{1\}$. We say that $\BB\colon\R^3\to\R^3$ is $\SO$-equivariant if and only if $g\BB=\BB(g\cdot)$ for every $g\in\SO$.

Let $$\cF:=\Big\{\UU\colon\R^3\to\R^3:\UU(x_1,x_2,x_3)=\frac{b(x)}{r}\Bigl(\begin{smallmatrix}
-x_2\\
x_1\\
0
\end{smallmatrix}\Bigr)\hbox{ for some }\cO\hbox{-invariant } b\colon\R^3\to\R\Big\}.$$
Then $\DF:=\cD^{1,2}(\R^3,\R^3)\cap\cF$ is a closed subspace of $\cD^{1,2}(\R^3,\R^3)$ and note that every $\UU\in\cF$ is $\SO$-equivariant.
	
	The main result of this section is the following.
	
	\begin{Th}\label{ScalVec}
		Let $f$ satisfy (F1) and $|f(x,u)|\leq C|u|^5$ for a.e. $x\in\R^3$, every $u\in\R$ and some constant $C>0$. Let $h\colon\R^3\times\R^3\to\R^3$ be such that $h(x,\alpha w)=f(x,\alpha)w$ for a.e. $x\in\R^3$, for every $\alpha\in\R$ and for every $w\in\R^3$ with $|w|=1$. Suppose that 
		$\UU$ and $u$  satisfy \eqref{eq:formulauU} for a.e. $x\in\R^3$. Then
		$\UU\in\DF$ if and only if $u\in X_{\cO}$ and, in such a case, $\div(\UU)=0$ and
		$\J(u)=\cE(\UU)$. Moreover 
		$u\in X_{\cO}$ is a solution to \eqref{eqS} with $a=1$ 
		if and only if $\UU\in\D_\cF$ is a solution to \eqref{eqV}. 
	\end{Th}
	
	\begin{Lem}\label{OApprox}
		If $\UU\in\mathcal{D}^{1,2}(\R^3,\R^3)$ is $\SO$-equivariant, then there exists $(\UU_n)\subset\cC^\infty_0(\R^3,\R^3)$ such that $\UU_n$ is $\SO$-equivariant and $|\nabla\UU_n-\nabla \UU|_{2}\to 0$ as $n\to\infty$.
	\end{Lem}
	\begin{proof}
		Since $\UU\in\cD^{1,2}(\R^3,\R^3)$, there exists $(\VV_n)\subset\cC^\infty_0(\R^3,\R^3)$ such that $|\nabla \VV_n-\nabla\UU|_{2}\to 0$ as $n\to\infty$. Let $$\UU_n(x):=\int_{\SO}g^{-1}\VV_n(gx)\,d\mu(g)=\int_{\SO}g^T\VV_n(gx)\,d\mu(g),$$ where $\mu$ is the Haar measure of $\SO$ (note that $\SO$ is compact).
		
		For every $e\in\SO$ we have
		\[
		\UU_n(ex)=\int_{\SO}g^T\VV_n(gex)\,d\mu(g)=e\int_{\SO}g'^T\VV_n(g'x)\,d\mu(g')=e\UU_n(x),
		\]
		i.e., $\UU_n$ is $\SO$-equivariant. Moreover, $\UU_n\in\cC^\infty_0(\R^3,\R^3)$ because so does $\VV_n$.
		
		First we prove that $|\UU_n-\UU|_6\to 0$. From Jensen's inequality there holds
		\[\begin{split}
		|\UU_n-\UU|_6^6 & =\int_{\mathbb{R}^3}\left|\int_{\SO}g^T\VV_n(gx)-\UU(x)\,d\mu(g)\right|^6dx\leq\int_{\SO}\int_{\mathbb{R}^3}\left|g^T\VV_n(gx)-\UU(x)\right|^6dx\,d\mu(g)\\ & =\int_{\SO}\int_{\mathbb{R}^3}\left|g^T\VV_n(gx)-g^T\UU(gx)\right|^6dx\,d\mu(g)=\int_{\SO}\int_{\mathbb{R}^3}\left|\VV_n(gx)-\UU(gx)\right|^6dx\,d\mu(g)\\ &
		=\int_{\SO}|\VV_n-\UU|_6^6\,d\mu(g)=|\VV_n-\UU|_6^6\to 0
		\end{split}\]
		as $n\to\infty$.
		
		Finally, $(\UU_n)$ is a Cauchy sequence in $\cD^{1,2}(\R^3,\R^3)$ because
		\[
		|\nabla \UU_n-\nabla \UU_m|_2\le|\nabla \VV_n-\nabla\VV_m|_2\to 0
		\]
		as $n,m\to\infty$ and we conclude.
	\end{proof}
	 
	\begin{Prop}\label{Density}
		Let 
		\begin{eqnarray*}
		&&\cH:=\big\{\UU\colon\R^3\to\R^3:|\UU(x)|\leq C|(x_1,x_2)|\hbox{ for some }C>0\\
		&& \hspace{3.8cm}\hbox{ uniformly with respect to } x_3 \hbox{ as }(x_1,x_2)\to 0\big\},
		\end{eqnarray*}
		 and set $\R^3_*:=\R^3\setminus(\{0\}\times\{0\}\times\R)$. Then $$\DF=\overline{\cC_0(\R^3,\R^3)\cap\cC^\infty(\R^3_*,\R^3)\cap\cH\cap\DF},$$ where the closure is taken in $\D^{1,2}(\R^3,\R^3)$.
	\end{Prop}
	\begin{proof}
		The inclusion `$\supset$' is obvious since $\DF$ is closed. Now let $\UU\in\DF$. Since $\UU$ is $\SO$-equivariant, in view of Lemma \ref{OApprox} there exists an $\SO$-equivariant sequence $(\UU_n)\subset\cC^\infty_0(\R^3,\R^3)$ such that $\UU_n=(\UU_n^1,\UU_n^2,\UU_n^3)\to\UU$ in $\D^{1,2}(\R^3,\R^3)$.
		
		As in \cite[Lemma 1]{AzzBenDApFor}, for every $n$ there exist $\SO$-equivariant  $\UU_{\rho,n}$, $\UU_{\tau,n}$, $\UU_{\zeta,n}\in\D^{1,2}(\R^3,\R^3)$ such that for every $x\in\R^3_*$
		\begin{itemize}
			\item $\UU_{\rho,n}(x)$ is the projection of $\UU_n(x)$ onto $\span \{(x_1,x_2,0)\}$,
			\item $\UU_{\tau,n}(x)$ is the projection of $\UU_n(x)$ onto $\span\{(-x_2,x_1,0)\}$,
			\item $\UU_{\zeta,n}(x)=\bigl(0,0,\UU_n^3(x)\bigr)$ is the projection of $\UU_n(x)$ onto $\span\{(0,0,1)\}$.
		\end{itemize}
		In particular, $\UU_{\rho,n},\UU_{\tau,n},\UU_{\zeta,n}\in\cC^\infty(\R^3_*,\R^3)$, they vanish outside a sufficiently large ball in $\R^3$ (in fact, $\UU_{\zeta,n}\in\cC^\infty_0(\R^3,\R^3)$) and $\UU_n(x)=\UU_{\rho,n}(x)+\UU_{\tau,n}(x)+\UU_{\zeta,n}(x)$ for every $x\in\R^3_*$. Moreover, $\nabla \UU_{\rho,n}(x)$, $\nabla \UU_{\tau,n}(x)$, $\nabla \UU_{\zeta,n}(x)$ are orthogonal in $\R^{9}$ for every $x\in\R^3_*$.
		
		Then we infer that $\UU_{\tau,n}\to\UU$ in $\cD^{1,2}(\R^3,\R^3)$ so we just need to prove that $\UU_{\tau,n}\in\cC_0(\R^3,\R^3)\cap\cH$ for every $n$.
		
		Since $\UU_n$ is $\SO$-equivariant, for every $g\in\SO$ and every $x_3\in\R$ we have $$g\UU_n(0,0,x_3)=\UU_n\bigl(g(0,0,x_3)\bigr)=\UU_n(0,0,x_3),$$ which implies $\UU_n^1(0,0,x_3)=\UU_n^2(0,0,x_3)=0$, so that $\UU_n^1=\UU_n^2\equiv 0$ on $\{0\}\times\{0\}\times\R$.
		
		Observe that for every $x\in\R^3_*$ we have $$\UU_{\rho,n}(x)=\frac{\langle\UU_n,(x_1,x_2,0)\rangle}{|(x_1,x_2)|^2}\Bigl(\begin{smallmatrix}
		x_1\\
		x_2\\
		0
		\end{smallmatrix}\Bigr)\text{ and }\UU_{\tau,n}(x)=\frac{\langle\UU_n,(-x_2,x_1,0)\rangle}{|(x_1,x_2)|^2}\Bigl(\begin{smallmatrix}
		-x_2\\
		x_1\\
		0
		\end{smallmatrix}\Bigr)$$  
		and it follows from the uniform continuity of $\UU_n$ that for every $x_3\in\R$ $$\lim_{(x_1,x_2)\to 0}\UU_{\rho,n}(x)=\lim_{(x_1,x_2)\to 0}\UU_{\tau,n}(x)=0$$ uniformly with respect to $x_3$. Hence we can extend $\UU_{\rho,n}$ and  $\UU_{\tau,n}$ to $\R^3$ by setting them equal to $0$ on $\{0\}\times\{0\}\times\R$ and get that $\UU_{\rho,n}$, $\UU_{\tau,n}\in\cC_0(\R^3,\R^3)$ and $\UU_n(x)=\UU_{\rho,n}(x)+\UU_{\tau,n}(x)+\UU_{\zeta,n}(x)$ for every $x\in\R^3$.
		
		To prove that $\UU_{\rho,n}+\UU_{\tau,n}\in\cH$, first we notice that $\big(\UU_{n}-\UU_{\zeta,n}\big)\in\cC_0^\infty(\R^3,\R^3)$ and
		\begin{eqnarray*}
		\big(\UU_{\rho,n}+\UU_{\tau,n}\big)(x)&=&\big(\UU_{n}-\UU_{\zeta,n}\big)(x)=\big(\UU_{n}-\UU_{\zeta,n}\big)(0,0,x_3)\\
		&&+\nabla \big(\UU_{n}-\UU_{\zeta,n}\big)(0,0,x_3)(x_1,x_2,0)^T+o(|(x_1,x_2)|)\\
		&=&\nabla \big(\UU_{n}-\UU_{\zeta,n}\big)(0,0,x_3)(x_1,x_2,0)^T+o(|(x_1,x_2)|)
		\end{eqnarray*} 
		as $(x_1,x_2)\to 0$, hence $\UU_{\rho,n}+\UU_{\tau,n}\in\cH$. Finally, note that $|\UU_{\tau,n}|\le|\UU_{\rho,n}+\UU_{\tau,n}|$ and we conclude.
	\end{proof}

From now on we assume that $\UU$ and $u$  satisfy \eqref{eq:formulauU} for a.e. $x\in\R^3$.
\begin{Lem}\label{DivFree}
$\UU\in\DF$ if and only if $u\in X_{\cO}$.  In such a case, $\div(\UU)=0$ and
$\J(u)=\cE(\UU)$, where $f$ and $h$ satisfy the assumptions of Theorem \ref{ScalVec}.
\end{Lem}
\begin{proof}
Suppose that $u\in X_\cO$ and
let $\UU$ be of the form \eqref{eq:formulauU}. We show that
the pointwise a.e. gradient of $\UU$ in $\R^3$ is also the distributional gradient of $\UU$ in $\R^3$. Indeed, for the derivative along $x_1$ of the first component of $\UU$  we have
\[\begin{split}
\int_{\R^3}u(x)\frac{-x_2}{\sqrt{x_1^2+x_2^2}}\partial_{x_1}\phi(x)\,dx&=
\int_{\R^3}\biggl(\partial_{x_1}u(x)\frac{x_2}{\sqrt{x_1^2+x_2^2}}\phi(x)-u(x)\frac{x_1x_2}{(x_1^2+x_2^2)^\frac{3}{2}}\phi(x)\biggr)\,dx\\
&=-\int_{\R^3}\partial_{x_1}\biggl(u(x)\frac{-x_2}{\sqrt{x_1^2+x_2^2}}\biggr)\phi(x)\,dx<\infty
\end{split}\]
for every $\phi\in\cC_0^{\infty}(\R^3)$, since $\displaystyle\int_{\R^3}u(x)\frac{x_1x_2}{(x_1^2+x_2^2)^\frac{3}{2}}\phi(x)\,dx<\infty$ for $u\in X$.	
For the derivative along $x_1$ of the second component of $\UU$ similarly we get
\begin{eqnarray*}
	&&\int_{\R^3}u(x)\frac{x_1}{\sqrt{x_1^2+x_2^2}}\partial_{x_1}\phi(x)\,dx\\
	&&=-\int_{\R^3}\biggl(\partial_{x_1}u(x)\frac{x_1}{\sqrt{x_1^2+x_2^2}}+u(x)\biggl(\frac{1}{\sqrt{x_1^2+x_2^2}}-\frac{x_1^2}{(x_1^2+x_2^2)^\frac{3}{2}}\biggr)\biggr)\phi(x)\,dx\\
	&&=-\int_{\R^3}\partial_{x_1}\biggl(u(x)\frac{x_1}{\sqrt{x_1^2+x_2^2}}\biggr)\phi(x)\,dx<\infty
\end{eqnarray*}
for every $\phi\in\cC_0^{\infty}(\R^3)$.
The remaining cases are similar. 

Now observe that $\UU=(\UU_1,\UU_2,\UU_3)\in L^6(\R^3,\R^3)\cap\cF$. Moreover, $$\partial_{x_1}\UU_1(x)=\partial_{x_1}u(x)\frac{x_2}{\sqrt{x_1^2+x_2^2}}-u(x)\frac{x_1x_2}{(x_1^2+x_2^2)^\frac{3}{2}}\in L^2(\R)$$ and $$\partial_{x_1}\UU_2(x)=-\partial_{x_1}u(x)\frac{x_1}{\sqrt{x_1^2+x_2^2}}-u(x)\biggl(\frac{1}{\sqrt{x_1^2+x_2^2}}-\frac{x_1^2}{(x_1^2+x_2^2)^\frac{3}{2}}\biggr)\in L^2(\R),$$ since $u\in X$.
Again, the remaining cases are similar and we infer that $\UU\in\cD_\cF$.

Now suppose that $\UU\in\DF$ and, due to Proposition \ref{Density}, let $(\BB_n)\subset\cC_0(\R^3,\R^3)\cap\cC^\infty(\R^3_*,\R^3)\cap\cH\cap\DF$ such that $\lim_n|\nabla\BB_n-\nabla\UU|_2=0$ and let $(b_n)$ be $\cO$-invariant such that $\BB_n$ and $b_n$ satisfy formula \eqref{eq:formulauU}.

We prove that $b_n\in X_{\cO}$. Since $|\BB_n|=|b_n|$, of course $b_n\in\cC_0(\R^3)\cap\cC^{\infty}(\R^3_*)\subset L^6(\R^3)$ and $|b_n(x)|\le C|(x_1,x_2)|$ for some $C>0$ uniformly with respect to $x_3$ as $(x_1,x_2)\to0$, therefore
\[
\int_{\R^3}\frac{b_n^2}{r^2}\,dx<\infty.
\]
Moreover, $$L^2(\R^3,\R^{3\times 3})\ni\nabla\BB_n(x)=\frac{1}{\sqrt{x_1^2+x_2^2}}\Bigl(\begin{smallmatrix}
-x_2\\
x_1\\
0
\end{smallmatrix}\Bigr)\nabla b_n(x)^T+\frac{b_n(x)}{(x_1^2+x_2^2)^\frac{3}{2}}\begin{pmatrix}
x_1x_2 & -x_1^2 & 0\\
x_2^2 & -x_1x_2 & 0\\
0 & 0 & 0
\end{pmatrix}$$ and the second summand is square-summable because $$\Bigg|\frac{1}{(x_1^2+x_2^2)^\frac{3}{2}}\begin{pmatrix}
x_1x_2 & -x_1^2 & 0\\
x_2^2 & -x_1x_2 & 0\\
0 & 0 & 0
\end{pmatrix}\Bigg|_{\R^{3\times 3}}=\frac{1}{(x_1^2+x_2^2)^\frac{3}{2}}\Bigg|\begin{pmatrix}
x_1\\
x_2\\
0
\end{pmatrix}\begin{pmatrix}
x_2 & -x_1 & 0
\end{pmatrix}\Bigg|_{\R^{3\times 3}}=\frac{1}{\sqrt{x_1^2+x_2^2}},$$
where $|\cdot|_{\R^{3\times 3}}$ stands for the matrix norm in $\R^{3\times 3}$. It follows that $\nabla b_n\in L^2(\R^3,\R^3)$, thus $b_n\in X_{\cO}$.

Since $\lim_n|b_n-u|_6=\lim_n|\BB_n-\UU|_6=0$, it is enough to prove that $(b_n)$ is a Cauchy sequence in $X$, therefore we compute
\[\begin{split}
\|b_n-b_m\|^2 & =\int_{\R^3}\langle\nabla(b_n-b_m),\nabla(b_n-b_m)\rangle+\frac{(b_n-b_m)(b_n-b_m)}{r^2}\,dx\\
& =\int_{\R^3}\langle\nabla(\BB_n-\BB_m),\nabla(\BB_n-\BB_m)\rangle\,dx=|\nabla(\BB_n-\BB_m)|_2^2\to0
\end{split}\]
as $n,m\to\infty$.

Next, observe that $\div(\BB_n(x))=0$ for every $x\in\R^3_*$. It follows that, up to a subsequence, $\div(\UU(x))=\lim_n\div(\BB_n(x))=0$ for a.e. $x\in\R^3$ and recall that the pointwise a.e. divergence of $\UU$ is also the distributional divergence of $\UU$.

Finally, observe that if $u\in X_{\cO}$ and $\UU\in \cD_\cF$ satisfy \eqref{eq:formulauU} a.e. on $\R^3$, then $\|u\|^2=|\nabla\UU|_2^2=|\curlop \UU|_2^2$ and $F(x,u(x))=H(x,\UU(x))$ for a.e. $x\in\R^3$.
	\end{proof}
	
	\begin{proof}[Proof of Theorem \ref{ScalVec}]
		The first part follows directly from Lemma \ref{DivFree}. Recall (cf. \cite[Section 2]{AzzBenDApFor}) that if $\UU\in\cD^{1,2}(\R^3,\R^3)$ is $\SO$-equivariant, then $\UU\in\DF$ if and only if $\UU$ is invariant with respect to the action
		\[
		\cS(\UU)=\cS(\UU_\rho+\UU_\tau+\UU_\zeta):=-\UU_\rho+\UU_\tau-\UU_\zeta.
		\]
		Recall also that the functional $\cE$ defined in \eqref{eq:I} is invariant under this action.
		
		Let $\VV\in\DF$ and $v\in X_{\cO}$ satisfy \eqref{eq:formulauU} and note that, arguing as in Lemma \ref{DivFree},
		\[
		\int_{\R^3}\langle\curlop\UU,\curlop\VV\rangle\,dx
		=\int_{\R^3}\langle\nabla\UU,\nabla\VV\rangle\,dx=\int_{\R^3}\langle\nabla u,\nabla v\rangle+\frac{uv}{r^2}\,dx
		\]
		and
		\[\begin{split}
		\int_{\R^3}\langle h\bigl(x,\UU(x)\bigr),\VV(x)\rangle\,dx&=\int_{\R^3}\Big\langle h\biggl(x,\frac{u}{r}\Bigl(\begin{smallmatrix}
		-x_2\\
		x_1\\
		0
		\end{smallmatrix}\Bigr)\biggr),\frac{v(x)}{r}\Bigl(\begin{smallmatrix}
		-x_2\\
		x_1\\
		0
		\end{smallmatrix}\Bigr)\Big\rangle\,dx\\
		&=\int_{\R^3}\Big\langle f\bigl(x,u(x)\bigr)\frac{1}{r}\Bigl(\begin{smallmatrix}
		-x_2\\
		x_1\\
		0
		\end{smallmatrix}\Bigr),\frac{v(x)}{r}\Bigl(\begin{smallmatrix}
		-x_2\\
		x_1\\
		0
		\end{smallmatrix}\Bigr)\Big\rangle\,dx\\
		&=\int_{\R^3}f\bigl(x,u(x)\bigr)v(x)\,dx.\qedhere
		\end{split}\]
	\end{proof}

\section{The critical case}\label{sec:crit}

In this section we prove Theorem \ref{critV}. Recall that in this context $N=3$ (hence $K=2$),
\begin{eqnarray*}
	\cE(\UU)&=&\frac12\int_{\R^3}|\curlop\UU|^2\,dx-\frac{1}{6}\int_{\R^3}|\UU|^6\,dx,\\
	\cJ(u)&=&\frac12\int_{\R^3}|\nabla u|^2+\frac{1}{r^2}|u|^2\,dx-\frac{1}{6}\int_{\R^3}|u|^6\,dx.
\end{eqnarray*}

Let $\pi\colon\mbS^3\setminus\{Q\}\to\R^3$ be the stereographic projection, where $Q=(0,0,0,1)$ is the north pole, and let
\[
\vphi\colon x\in\R^3\mapsto\sqrt{\frac{2}{|x|^2+1}}\in\R.
\]
Explicitly,
\[
\pi(\xi)=\frac{1}{1-\xi_4}(\xi_1,\xi_2,\xi_3),\quad\xi=(\xi_1,\xi_2,\xi_3,\xi_4),
\]
and
\[
\pi^{-1}(x)=\frac{1}{|x|^2+1}(2x_1,2x_2,2x_3,|x|^2-1),\quad x= (x_1,x_2,x_3).
\]

Recall that $\widetilde{g}=\bigl(\begin{smallmatrix}
	g & 0\\
	0 & 1
\end{smallmatrix}\bigr)$ for $g\in\SO(2)$, $\SO=\{\widetilde{g}:g\in\SO(2)\}$ and $\cD_{\SO(2)\times\SO(2)}$ is the subspace of $\cD^{1,2}(\R^3,\R^3)$ of $\SO(2)\times\SO(2)$-symmetric vector fields according to Definition \ref{def:sym}.

\begin{Lem}\label{equisym}
	If $\UU\in\cD^{1,2}(\R^3,\R^3)$ is $\SO(2)\times\SO(2)$-symmetric, then $\UU$ is $\SO$-equivariant.
\end{Lem}
\begin{proof}
	Let $g_1\in\SO(2)$ and define $g:=(g_1,I_2)\in\SO(2)\times\SO(2)$, where $I_2\in\SO(2)$ is the identity matrix. Note that $$g\pi^{-1}(x)=\pi^{-1}(\widetilde{g_1}x)$$ for every $x\in\R^3$, therefore
	\[\begin{split}
		\widetilde{g_1}\UU(x)&=\frac{\varphi(x)}{\varphi\Bigl(\pi\bigl(g\pi^{-1}(x)\bigr)\Bigr)}\UU\Bigl(\pi\bigl(g\pi^{-1}(x)\bigr)\Bigr)\\
		&=\frac{\varphi(x)}{\varphi(\widetilde{g_1}x)}\UU(\widetilde{g_1}x)=\UU	(\widetilde{g_1}x).\qedhere
	\end{split}\]
\end{proof}

\begin{Lem}\label{compact}
	The embedding
	$\cD_{\SO(2)\times\SO(2)}\subset L^6(\R^3,\R^3)$ is compact.
\end{Lem}
\begin{proof}
	For every $\UU\in\cD_{\SO(2)\times\SO(2)}$ define $\VV(\xi):=\frac{\UU(\pi(\xi))}{\vphi(\pi(\xi))}$ for $\xi\in\mathbb{S}^3\setminus\{Q\}$. We note that $\VV\in H^1(\mathbb{S}^3,\R^3)$ and  similarly as in \cite[Lemma 3.1]{ClPis} $|\nabla \UU|_2=\|\VV\|_{H^1(\mathbb{S}^3,\R^3)}$ and $|\UU|_{6}=|\VV|_{6}$,
	where 
	$$\|\VV\|^2_{H^1(\mathbb{S}^3,\R^3)}=\int_{\mathbb{S}^3}|\nabla_g \VV|^2+\frac34|\VV|^2\,dV_g$$ is the norm in 
	$H^1(\mathbb{S}^3,\R^3)$ and $\nabla_g$ is the gradient on $\mathbb{S}^3$.
	Therefore $\UU\mapsto\VV$ is a linear isometric isomorphism between $\cD^{1,2}(\R^3,\R^3)$ and $H^1(\mathbb{S}^3,\R^3)$ and between $L^6(\R^3,\R^3)$ and $L^6(\mbS^3,\R^3)$. Note that, since $\UU$ is $\SO(2)\times\SO(2)$-symmetric, then $\VV(g\xi)=\widetilde{g_1}\VV(\xi)$ for every $g=(g_1,g_2)\in\SO(2)\times\SO(2)$ and, consequently, $|\VV|$ is $\SO(2)\times\SO(2)$-invariant, or equivalently $\cO(2)\times\cO(2)$-invariant. 
	
	Let $(\UU_n)\subset\cD_{\SO(2)\times\SO(2)}$ such that $\UU_n\rightharpoonup0$ in $\cD_{\SO(2)\times\SO(2)}$. Then $\VV_n\rightharpoonup0$ in $H^1(\mathbb{S}^3,\R^3)$ and, up to a subsequence, $\VV\to0$ a.e. in $\mathbb{S}^3$; this implies that $|\VV_n|\rightharpoonup0$ in $H^1(\mathbb{S}^3)$ and so, in view of \cite[Lemma 5]{Ding}, $|\VV_n|\to0$ in $L^6(\mathbb{S}^3)$. Hence $\VV_n\to0$ in $L^6(\mathbb{S}^3,\R^3)$ and so $\UU_n\to0$ in $L^6(\R^3,\R^3)$.
\end{proof}

For $\UU\in\cD^{1,2}(\R^3,\R^3)$ recall from the proof of Proposition \ref{Density} the definition of $\UU_\rho$, $\UU_\tau$, and $\UU_\zeta$.

\begin{Lem}\label{lem:decom}
	If $\UU\in\cD_{\SO(2)\times\SO(2)}$, then $\UU_\rho,\UU_\tau,\UU_\zeta\in\cD_{\SO(2)\times\SO(2)}$.
\end{Lem}
\begin{proof}
	We begin proving that $\UU_\tau\in\cD_{\SO(2)\times\SO(2)}$. For every matrix $A$, let $A^T$ denote its transpose.
	
	Let $\alpha_i\in\R$, $g_i=\bigl(\begin{smallmatrix}
		\cos\alpha_i & -\sin\alpha_i\\
		\sin\alpha_i & \cos\alpha_i
	\end{smallmatrix}\bigr)\in\SO(2)$, $i=1,2$, and set $g=\bigl(\begin{smallmatrix}
		g_1 & 0\\
		0 & g_2
	\end{smallmatrix}\bigr)$. We want to prove that
	\begin{equation}\label{eq:tau}
		\frac{\varphi(x)}{\varphi\Bigl(\pi\bigl(g\pi^{-1}(x)\bigr)\Bigr)}\UU_\tau\Bigl(\pi\bigl(g\pi^{-1}(x)\bigr)\Bigr)=\widetilde{g_1}\UU_\tau(x)
	\end{equation}
	provided
	\[
	\frac{\varphi(x)}{\varphi\Bigl(\pi\bigl(g\pi^{-1}(x)\bigr)\Bigr)}\UU\Bigl(\pi\bigl(g\pi^{-1}(x)\bigr)\Bigr)=\widetilde{g_1}\UU(x).
	\]
	
	We compute the two sides of \eqref{eq:tau} separately. We use the convention that $\R^3=\R^{3\times1}$ and treat the scalar product in $\R^3$ as matrix multiplication.
	
	As for the right-hand side we have
	\[\begin{split}
		\widetilde{g_1}\UU_\tau(x)&=\frac{\widetilde{g_1}\Bigl(\begin{smallmatrix}
				-x_2\\
				x_1\\
				0
			\end{smallmatrix}\Bigr)\UU^T(x)\Bigl(\begin{smallmatrix}
				-x_2\\
				x_1\\
				0
			\end{smallmatrix}\Bigr)}{x_1^2+x_2^2}=\frac{\Bigl(\begin{smallmatrix}
				-x_2\cos\alpha_1-x_1\sin\alpha_1\\
				-x_2\sin\alpha_1+x_1\cos\alpha_1\\
				0
			\end{smallmatrix}\Bigr)\UU^T(x)\Bigl(\begin{smallmatrix}
				-x_2\\
				x_1\\
				0
			\end{smallmatrix}\Bigr)}{x_1^2+x_2^2}\\
		&=\frac{-x_2\UU_1(x)+x_1\UU_2(x)}{x_1^2+x_2^2}\Bigl(\begin{smallmatrix}
			-x_2\cos\alpha_1-x_1\sin\alpha_1\\
			-x_2\sin\alpha_1+x_1\cos\alpha_1\\
			0
		\end{smallmatrix}\Bigr).
	\end{split}\]
	
	Let us write $\pi=\Bigl(\begin{smallmatrix}
		\pi_1\\
		\pi_2\\
		\pi_3
	\end{smallmatrix}\Bigr)$. As for the left-hand side we have
	\[\begin{split}
		\frac{\varphi(x)}{\varphi\Bigl(\pi\bigl(g\pi^{-1}(x)\bigr)\Bigr)}\UU_\tau\Bigl(\pi\bigl(g\pi^{-1}(x)\bigr)\Bigr)\\
		=\frac{\varphi(x)}{\varphi\Bigl(\pi\bigl(g\pi^{-1}(x)\bigr)\Bigr)}\frac{\biggl(\begin{smallmatrix}
				-\pi_2(g\pi^{-1}(x))\\
				\pi_1(g\pi^{-1}(x))\\
				0
			\end{smallmatrix}\biggr)\UU^T\Bigl(\pi\bigl(g\pi^{-1}(x)\bigr)\Bigr)\biggl(\begin{smallmatrix}
				-\pi_2(g\pi^{-1}(x))\\
				\pi_1(g\pi^{-1}(x))\\
				0
			\end{smallmatrix}\biggr)}{\pi_1^2\bigl(g\pi^{-1}(x)\bigr)+\pi_2^2\bigl(g\pi^{-1}(x)\bigr)}\\
		=\frac{\biggl(\begin{smallmatrix}
				-\pi_2(g\pi^{-1}(x))\\
				\pi_1(g\pi^{-1}(x))\\
				0
			\end{smallmatrix}\biggr)\UU^T(x)\widetilde{g_1}^T\biggl(\begin{smallmatrix}
				-\pi_2(g\pi^{-1}(x))\\
				\pi_1(g\pi^{-1}(x))\\
				0
			\end{smallmatrix}\biggr)}{\pi_1^2\bigl(g\pi^{-1}(x)\bigr)+\pi_2^2\bigl(g\pi^{-1}(x)\bigr)}.
	\end{split}\]
	
	Let us compute
	\[
	g\pi^{-1}(x)=\frac{1}{|x|^2+1}\Biggl(\begin{smallmatrix}
		2x_1\cos\alpha_1-2x_2\sin\alpha_1\\
		2x_1\sin\alpha_1+2x_2\cos\alpha_1\\
		2x_3\cos\alpha_2-(|x|^2-1)\sin\alpha_2\\
		2x_3\sin\alpha_2+(|x|^2-1)\cos\alpha_2
	\end{smallmatrix}\Biggr),
	\]
	\[
	\pi\bigl(g\pi^{-1}(x)\bigr)=\frac{1}{|x|^2+1-2x_3\sin\alpha_2-(|x|^2-1)\cos\alpha_2}\biggl(\begin{smallmatrix}
		2x_1\cos\alpha_1-2x_2\sin\alpha_1\\
		2x_1\sin\alpha_1+2x_2\cos\alpha_1\\
		2x_3\cos\alpha_2-(|x|^2-1)\sin\alpha_2
	\end{smallmatrix}\biggr),
	\]
	\[
	\UU^T(x)\widetilde{g_1}^T=\biggl(\begin{smallmatrix}
		\UU_1(x)\cos\alpha_1-\UU_2(x)\sin\alpha_1\\
		\UU_1(x)\sin\alpha_1+\UU_2(x)\cos\alpha_1\\
		\UU_3(x)
	\end{smallmatrix}\biggr)^T,
	\]
	\[
	\UU^T(x)\widetilde{g_1}^T\biggl(\begin{smallmatrix}
		-\pi_2(g\pi^{-1}(x))\\
		\pi_1(g\pi^{-1}(x))
		\\
		0
	\end{smallmatrix}\biggr)=\frac{2\bigl(-x_2\UU_1(x)+x_1\UU_2(x)\bigr)}{|x|^2+1-2x_3\sin\alpha_2-(|x|^2-1)\cos\alpha_2},
	\]
	and
	\[
	\pi_1^2\bigl(g\pi^{-1}(x)\bigr)+\pi_2^2\bigl(g\pi^{-1}(x)\bigr)=\frac{4x_1^2+4x_2^2}{\bigl(|x|^2+1-2x_3\sin\alpha_2-(|x|^2+1)\cos\alpha_2\bigr)^2},
	\]
	so for the left-hand side we have
	\[
	\frac{\biggl(\begin{smallmatrix}
			-\pi_2(g\pi^{-1}(x))\\
			\pi_1(g\pi^{-1}(x))\\
			0
		\end{smallmatrix}\biggr)\UU^T(x)\widetilde{g_1}^T\biggl(\begin{smallmatrix}
			-\pi_2(g\pi^{-1}(x))\\
			\pi_1(g\pi^{-1}(x))\\
			0
		\end{smallmatrix}\biggr)}{\pi_1^2\bigl(g\pi^{-1}(x)\bigr)+\pi_2^2\bigl(g\pi^{-1}(x)\bigr)}=\frac{-x_2\UU_1(x)+x_1\UU_2(x)}{x_1^2+x_2^2}\Bigl(\begin{smallmatrix}
		-x_2\cos\alpha_1-x_1\sin\alpha_1\\
		-x_2\sin\alpha_1+x_1\cos\alpha_1\\
		0
	\end{smallmatrix}\Bigr)
	\]
	and \eqref{eq:tau} holds.
	
	Similar computations hold for $\UU_\rho$. Finally, $\UU_\zeta=\UU-\UU_\rho-\UU_\tau\in\cD_{\SO(2)\times\SO(2)}$.
\end{proof}

Note that $\cE|_{\DF}=L|_{\DF}$ in view of Lemma \ref{DivFree}, where $L\colon\cD^{1,2}(\R^3,\R^3)\to\R$ is defined as $$L(\UU):=\frac12\int_{\R^3}|\nabla\UU|^2\,dx-\frac{1}{6}\int_{\R^3}|\UU|^6\,dx.$$ We set $\cY:=\cD_{\SO(2)\times\SO(2)}\cap\cF$.

\begin{Lem}\label{lem:inf}
	$\cY$ is infinite dimensional.
\end{Lem}
\begin{proof}
	Let $e=\bigl(\begin{smallmatrix}
		0 & -1\\
		1 & 0
	\end{smallmatrix}\bigr)\in\SO(2)$ and	
	\begin{eqnarray*}
		\cX&:=&\Big\{\UU\in \cD_{\SO(2)\times\SO(2)}: \UU(x_1,x_2,x_3)=\frac{u(x)}{r}\Bigl(\begin{smallmatrix}
			x_1\\
			x_2\\
			0
		\end{smallmatrix}\Bigr)\hbox{ for some }\cO\hbox{-invariant } u\colon\R^3\to\R\Big\},\\ \cZ&:=&\Big\{\UU\in \cD_{\SO(2)\times\SO(2)}:\UU(x_1,x_2,x_3)=u(x)\Bigl(\begin{smallmatrix}
			0\\
			0\\
			1
		\end{smallmatrix}\Bigr)\hbox{ for some }\cO\hbox{-invariant } u\colon\R^3\to\R\Big\}.
	\end{eqnarray*} 
	In order to prove that $\cY$ is infinite dimensional, we build an isomorphism between $\cX$ and $\cY$ and an isomorphism between $\cX$ and $\cZ$. The conclusion will follow from the fact that $\cD_{\SO(2)\times\SO(2)}$ is infinite dimensional and that, in view of Lemma \ref{lem:decom}, we get the following decomposition $\cD_{\SO(2)\times\SO(2)}=\cX\oplus\cY\oplus\cZ$.
	
	Indeed, for every $\UU\in\cX$ define $\widetilde{\UU}(x):=\UU(\widetilde{e}x)$. It is clear that $\widetilde{\UU}\in\cY$ and that $\UU\mapsto\widetilde{\UU}$ is an isomorphism.
	
	Now consider $\UU\in\cX$ and let $u\colon\R^3\to\R$ be $\cO$-invariant such that $\UU(x)=\frac{u(x)}{r}\Bigl(\begin{smallmatrix}
		x_1\\
		x_2\\
		0
	\end{smallmatrix}\Bigr)$. Define $\overline{\UU}(x):=u(x)\Bigl(\begin{smallmatrix}
		0\\
		0\\
		1
	\end{smallmatrix}\Bigr)$. By similar arguments to those used in the proof of Lemma \ref{DivFree} it is easy to check that $\overline{\UU}\in\cD^{1,2}(\R^3,\R^3)$. Finally, explicit computations show that $\overline{\UU}$ is $\SO(2)\times\SO(2)$-symmetric (hence $\overline{\UU}\in\cZ$) and of course $\UU\mapsto\overline{\UU}$ is an isomorphism.
\end{proof}

\begin{proof}[Proof of Theorem \ref{critV}]
	Lemma \ref{equisym} implies that $\cY\subset\cD_\cF$; moreover, $\cY$ is closed in $\cD^{1,2}(\R^3,\R^3)$ and infinite dimensional by Lemma \ref{lem:inf}. Since $\UU\mapsto\frac{\UU}{\vphi}\circ\pi$ is a linear isometry between $\cD^{1,2}(\R^3,\R^3)$ and $H^1(\mathbb{S}^3,\R^3)$ and between $L^6(\R^3,\R^3)$ and $L^6(\mbS^3,\R^3)$, one easily checks that $\cE|_{\DF}$ is invariant under the action of $\SO(2)\times\SO(2)$. Hence every $\UU\in\cY$ is a solution to \eqref{eq:crV} if and only if it is a critical point of $\cE|_\cY$.
	
	It is easy to see that there exists $\rho>0$ such that $\inf\{\cE(\UU):\UU\in\cY\text{ and }|\nabla\UU|_2=\rho\}>0$ and, in view of Lemma \ref{compact}, that $\cE|_\cY$ satisfies the Palais-Smale condition at every positive level.
	Let $E\subset\cY$ be a finite dimensional subspace. Then the norms $|\nabla(\cdot)|_2$ and $|\cdot|_6$ are equivalent in $E$. This implies that there exists $R=R(E)>0$ such that $\cE(\UU)\le0$ for every $\UU\in E$ with $|\UU|_6\ge R$. Hence the conclusion follows from \cite[Theorem 9.12]{Rabin} and the Palais principle of symmetric criticality \cite{Palais}. 
\end{proof}

\begin{proof}[Proof of Corollary \ref{critS}]
	The proof follows from Theorems \ref{critV} and  \ref{ScalVec}.
\end{proof}

	\section{The noncritical case}\label{sec:noncrit}
	
	In this section we prove Theorems \ref{ExMul} and \ref{ExMul2}. Throughout this section we assume $f$ satisfies (F1) and (F2). 
	The following lemma is proved in \cite[Proposition A.2]{MederskiZeroMass}.
	
	\begin{Lem}\label{Lions}
		Suppose that $(u_n)\subset \cD^{1,2}(\R^N)$ is bounded and $\cO$-invariant and for all $R>0$
		\begin{equation}\label{eq:LionsCond12}
		\lim_{n\to\infty}\sup_{z\in \R^{N-K}} \int_{B((0,z),R)} |u_n|^2\,dx=0.
		\end{equation}
		Then $$\int_{\R^N} \Phi(u_n)\, dx\to 0\quad\hbox{as } n\to\infty$$ for any continuous function $\Phi\colon\R\to [0,\infty)$ such that
		\begin{equation}\label{eq:LionsPhi}
		\displaystyle\lim_{s\to 0}\frac{\Phi(s)}{|s|^{2^*}}=\lim_{|s|\to\infty}\frac{\Phi(s)}{|s|^{2^*}}=0.
		\end{equation}
	\end{Lem}
	
We need the following results as well. 
	
\begin{Lem}\label{pqDec}
	Let $1\le p\le 2^*\le q<\infty$. If $u\in L^{2^*}(\R^N)$, then 
	$$|u\chi_{\{|u|\le 1\}}|_q^q,\;|u\chi_{\{|u|>1\}}|_p^p,\;|\{|u|>1\}|\le|u|_{2^*}^{2^*},$$
	where  $\chi$ denotes the characteristic function and $|\cdot|$ stands for the Lebesgue measure.
\end{Lem}
\begin{proof}
	Clearly 
	$$\int_{\R^N}|u|^q\chi_{\{|u|\le 1\}}\,dx\le\int_{\R^N}|u|^{2^*}\chi_{\{|u|\le1\}}\,dx\le|u|_{2^*}^{2^*}.$$
	Moreover, we have that $$|\{|u|>1\}|=\int_{\{|u|>1\}}1\,dx\le\int_{\{|u|>1\}}|u|^{2^*}\,dx\le|u|_{2^*}^{2^*}$$ and so $$\int_{\R^N}|u|^p\chi_{\{|u|>1\}}\,dx\le|u|_{2^*}^p|\{|u|>1\}|^\frac{2^*-p}{2^*}\le|u|_{2^*}^p|u|_{2^*}^{2^*-p}=|u|_{2^*}^{2^*}.\qedhere$$
\end{proof}

\begin{Lem}\label{Lions2}
	Suppose that $(u_n),(v_n)\subset \cD^{1,2}(\R^N)$ are bounded and $\cO$-invariant and  $(u_n)$ satisfies \eqref{eq:LionsCond12} for all $R>0$. Then $$\int_{\R^N}|f(x,v_n)u_n|\,dx\to 0$$ as $n\to\infty$.
\end{Lem}
\begin{proof}
	Let $1<p<2^*<q<\infty$ and define $\Phi(t):=\int_0^{|t|}\min\{s^{p-1},s^{q-1}\}\,ds$. Note that $\Phi$ satisfies \eqref{eq:LionsPhi}. (F2) implies that for every $\varepsilon>0$ there exists $C_\varepsilon>0$ such that for every $t\in\R$ and a.e. $x\in\R^N$ we have $|f(x,t)|\le\varepsilon|t|^{2^*-1}+C_\varepsilon|\Phi'(t)|$. Moreover
	\[
	\int_{\R^N}|\Phi'(v_n)u_n|\,dx=\int_{\R^N}|\Phi'(v_n)u_n|\chi_{\{|u_n|>1\}}\,dx+\int_{\R^N}|\Phi'(v_n)u_n|\chi_{\{|u_n|\le 1\}}\,dx=:A_n+B_n.
	\]
	Concerning the first integral $A_n$, Lemmas \ref{Lions} and \ref{pqDec} imply that, for some $C_1,C_2>0$,
	\[\begin{split}
	A_n&= \int_{\R^N}|v_n|^{p-1}\chi_{\{|v_n|>1\}}|u_n|\chi_{\{|u_n|>1\}}\,dx+\int_{\R^N}|v_n|^{q-1}\chi_{\{|v_n|\le 1\}}|u_n|\chi_{\{|u_n|>1\}}\,dx\\
	& \le\bigl(|v_n\chi_{\{|v_n|>1\}}|_p^{p-1}+||v_n|^{q-1}\chi_{\{|v_n|\le 1\}}|_{\frac{p-1}{p}}\bigr)|u_n\chi_{\{|u_n|>1\}}|_p\\
	& \le C_1\Bigl(|v_n\chi_{\{|v_n|>1\}}|_p^{p-1}+|v_n\chi_{\{|v_n|\le 1\}}|_q^\frac{q(p-1)}{p}\Bigr)\biggl(\irn\Phi(u_n)\,dx\biggr)^\frac{1}{p}\\
	& \le C_2\sup_k\|v_k\|^\frac{2^*(p-1)}{p}\biggl(\irn\Phi(u_n)\,dx\biggr)^\frac{1}{p}\to0
	\end{split}\]
	because $\bigl(|v_n|^{q-1}\bigr)^{\frac{p}{p-1}}\chi_{\{|v_n|\le 1\}}\le|v_n|^q\chi_{\{|v_n|\le 1\}}$.
	
	Finally, similar computations hold for the second integral $B_n$.
\end{proof}

In order to prove Theorem \ref{ExMul} we aim to use the abstract critical point theory from \cite[Section 3]{MeScSz}, in particular Theorems 3.3 and 3.5(b) therein. We need to prove that assumptions (I1)--(I8), (G), and $(M)_0^\beta$ for every $\beta>0$, which in our setting read as follows, are satisfied. For simplicity, we set 
$$\cI(u):=\irn F(x,u)\,dx\quad\hbox{for }u\in X_{\cO}$$ and $$\cN:=\bigl\{u\in X_{\cO}\setminus\{0\}:\cJ'(u)u=0\bigr\}$$
stands for the {\em Nehari constraint}, which need not be a manifold of class $\cC^1$; see \cite{MeScSz}. We list the required conditions:

\begin{itemize}
	\item [(I1)] $\cI\in\cC^1(X_{\cO},\R)$ and $\cI(u)\ge\cI(0)=0$ for every $u\in X_{\cO}$.
	\item [(I2)] $\cI$ is sequentially lower semicontinuous.
	\item [(I3)] If $u_n\to u$ and $\cI(u_n)\to\cI(u)$, then $u_n\to u$.
	\item [(I4)] $\|u\|+\cI(u)\to\infty$ as $\|u\|\to\infty$.
	\item [(I6)] There exists $r>0$ such that $\displaystyle\inf_{u\in X_{\cO},\|u\|=r}\cJ(u)>0$.
	\item [(I7)] $\displaystyle\frac{\cI(u_n)}{t_n^2}\to\infty$ if $t_n\to\infty$ and $u_n\to u_0\ne0$ as $n\to\infty$.
	\item [(I8)] $\displaystyle\frac{t^2-1}{2}\cI'(u)u+\cI(u)-\cI(tu)\le0$ for every $u\in\cN$ and every $t\ge0$.
	\item [(G)] $\Z^{N-K}$ is a group that acts on $X_{\cO}$ by isometries and such that for every $u\in X_{\cO}$, $(\Z^{N-K}*u)\setminus\{u\}$ is bounded away from $u$.
\end{itemize}

The $\Z^{N-K}$-action is given as follows: $z\ast u(x):=u(x+(0,z))$ for $z\in \Z^{N-K}$ and $u\in X_{\cO}$. $\Z^{N-K}\ast u$ is called the {\em orbit} of $u$ and if, in addition, $u$ is a critical point of $\J$, then $\Z^{N-K}\ast u$ is a {\em critical orbit}.

Note that (I1)--(I4) and (G) are obviously satisfied when (F4) holds and (I6) follows easily from (F2) and the embedding of $X_{\cO}$ into $L^{2^*}(\R^N)$. We have skipped (I5) from \cite[Section 3]{MeScSz}, since it is an empty condition. (I7), (I8), and the following variant of Cerami condition will be verified in the next lemmas.

\begin{itemize}
	\item [$(M)_0^\beta$]
	\begin{itemize}
		\item [(a)] There exists $M_\beta>0$ such that $\limsup_n\|u_n\|\le M_\beta$ for every $(u_n)\subset X_{\cO}$ such that $0\le\liminf_n\cJ(u_n)\le\limsup_n\cJ(u_n)\le\beta$ and $\lim_n(1+\|u_n\|)\cJ'(u_n)=0$.
		\item [(b)] If $\cJ$ has finitely many critical orbits, then there exists $m_\beta>0$ such that, if $(u_n),(v_n) \subset X_{\cO}$ are as above and $\|u_n-v_n\|<m_\beta$ for $n$ large, then $\liminf_n\|u_n-v_n\|=0$.
	\end{itemize}
\end{itemize}

\begin{Lem}\label{lem:I78}
(a) Suppose $f$ satisfies (F3). If $t_n\to\infty$ and $u_n\to u_0\in X_{\cO}\setminus\{0\}$ as $n\to\infty$, then $$\lim_n\frac{1}{t_n^2}\irn F(x,t_nu_n)\,dx=\infty.$$
(b) Suppose $f$ satisfies (F4). For every $u\in X_{\cO}$ and every $t\ge0$ $$\frac{t^2-1}{2}\irn f(x,u)u\,dx+\irn F(x,u)\,dx-\irn F(x,tu)\,dx\le0.$$
\end{Lem}
\begin{proof}
	$(a)$ Since $X_{\cO}$ is locally compactly embedded into $L^2(\R^N)$, up to a subsequence $u_n\to u_0\neq 0$ a.e. in $\rn$. Moreover, there exists $\eps>0$ such that $\limsup_n|\Om_n|>0$, where $\Om_n:=\{x\in\rn:|u_n(x)|\ge\eps\}$, for otherwise $u_n\to0$ in measure and consequently, up to a subsequence, a.e. in $\rn$.
	It follows from (F3) that
	\[
	\frac{1}{t_n^2}\irn F(x,t_nu_n)\,dx=\irn\frac{F(x,t_nu_n)}{t_n^2|u_n|^2}|u_n|^2\,dx\ge\eps^2\int_{\Om_n}\frac{F(x,t_nu_n)}{t_n^2|u_n|^2}\,dx\to\infty
	\]
	as $n\to\infty$.
	
	$(b)$ For fixed $x\in\rn$ and $u\in\R$ we prove that $\phi(t)\le0$ for every $t\ge0$, where $$\phi(t):=\frac{t^2-1}{2}f(x,u)u+F(x,u)-F(x,tu).$$ This is trivial for $u=0$, so suppose $u\ne0$.
	Note that $\phi(1)=0$, so it is enough to prove that $\phi$ is nondecreasing on $[0,1]$ and nonincreasing on $[1,\infty)$. This is the case in view of (F4) and because
	\[
	\phi'(t)=tf(x,u)u-f(x,tu)u=t|u|u\biggl(\frac{f(x,u)}{|u|}-\frac{f(x,tu)}{|tu|}\biggr)
	\]
	for $t>0$, therefore $\phi(t)\le0$ for every $t\ge0$ as $\phi\in\cC^1([0,\infty))$.
\end{proof}

The following lemma shows that $(M)_0^\beta$ holds for every $\beta>0$.

\begin{Lem}\label{lem:Mbeta}
	Suppose $f$ satisfies (F3) and (F4).\\
(a) For every $\beta>0$ there exists $M_\beta>0$ such that $\limsup_n\|u_n\|\le M_\beta$ for every $(u_n)\subset X_{\cO}$ such that $\cJ(u_n)\le\beta$ for $n$ large and $\lim_n(1+\|u_n\|)\cJ'(u_n)=0$.\\
(b) If the number of critical orbits of $\cJ$ is finite, then there exists $\kappa>0$ such that, if $(u_n),(v_n)\subset X_{\cO}$ are as above for some $\beta>0$ and $\|u_n-v_n\|<\kappa$ for $n$ large, then $\lim_n\|u_n-v_n\|=0$.
\end{Lem}

\begin{proof}
	$(a)$ Let $(u_n)\subset X_{\cO}$ as in the assumptions, suppose that $(u_n)$ is unbounded, and define $\bar{u}_n:=u_n/\|u_n\|$. Passing to a subsequence we may assume that $\|u_n\|\to\infty$ as $n\to\infty$.
	Similarly as in the proof of Lemma \ref{Lions2}, for any $\eps>0$ we find $C_\eps>0$ such that 
	$$\irn F(x,\bar{u}_n)\,dx\leq \eps |\bar{u}_n|_{2^*}^{2^*}+ C_\eps\Phi(\bar{u}_n)$$ 
	for every $n$, where $\Phi$ is defined therein. If $(\bar{u}_n)$ satisfies \eqref{eq:LionsCond12} for every $R>0$, hence the same holds for $(s\bar{u}_n)$ with $s\ge0$, then in view of Lemma \ref{Lions}
	$$\limsup_n\irn F(x,s\bar{u}_n)\,dx\leq \eps s^2\limsup_n|\bar{u}_n|_{2^*}^{2^*}$$ 
	for every $\eps>0$, hence $\lim_n\irn F(x,s\bar{u}_n)\,dx=0$.
	Then applying Lemma \ref{lem:I78}(b) with $u=u_n$ and $t=t_n:=s/\|u_n\|$ we obtain, up to a subsequence, that for every $s\ge0$
	\[\begin{split}
	\beta&\ge\limsup_n\cJ(u_n)\ge\limsup_n\cJ(s\bar{u}_n)-\lim_n\frac{t_n^2-1}{2}\cJ'(u_n)u_n\\
	&=\limsup_n\cJ(s\bar{u}_n)\ge Cs^2-\lim_n\irn F(x,s\bar{u}_n)\,dx=Cs^2
	\end{split}\]
	for some $C>0$, a contradiction. Hence, up to a subsequence, $\lim_n\int_{B((0,z_n),R)}|\bar{u}_n|^2\,dx>0$ for some $R>\sqrt{N-K}$ and $(z_n)\subset\Z^{N-K}$, where $z_n$ maximizes $z\mapsto\int_{B((0,z),R)}|u_n|^2\,dx$. Exploiting the $\Z^{N-K}$-invariance, we can assume that $$\int_{B(0,R)}|\bar{u}_n|^2\,dx\ge c$$ for $n$ large and some $c>0$.
	
	It follows that there exists $\bar{u}\in X_{\cO}\setminus\{0\}$ such that, up to a subsequence, $\bar{u}_n\rightharpoonup\bar{u}$ in $X$ and $\bar{u}_n\to\bar{u}$ in $L^2_\textup{loc}(\rn)$ and a.e. in $\rn$.
	
	From (F4), $2\cJ(u_n)-\cJ'(u_n)u_n=\irn f(x,u_n)u_n-2F(x,u_n)\,dx\ge0$, thus $\bigl(\cJ(u_n)\bigr)$ is bounded and due to (F3) we obtain
	\[
   o(1)=\frac{\cJ(u_n)}{\|u_n\|^2}\le C-\irn\frac{F(x,u_n)}{|u_n|^2}|\bar{u}_n|^2\,dx\to-\infty
	\]
	for some $C>0$, which is a contradiction.	
	This shows that $(u_n)$ is indeed bounded. If by contradiction there exists no upper bound $M_\beta$, then for every $k\in\mathbb{N}$ there exists $(u_n^k)\subset X_{\cO}$ as in the statement such that $\limsup_n\|u_n^k\|>k$ and it is easy to build a subsequence $(u_{n_k}^k)$ that is unbounded, again a contradiction.
	
	$(b)$ Assume that there are finitely many critical orbits of $\cJ$. From (G) we easily see that $$\kappa:=\inf\big\{\|u-v\|:u\ne v\text{ and }\cJ'(u)=\cJ'(v)=0\big\}>0.$$
	
	Let $(u_n),(v_n)$ be as in the assumptions of $(b)$. In view of $(a)$, they are bounded.
	
	If $\irn f(x,u_n)(u_n-v_n)\,dx$ or $\irn f(x,v_n)(u_n-v_n)\,dx$ do not converge to $0$, then in view of Lemma \ref{Lions2} and the $\Z^{N-K}$-invariance there exist $R>\sqrt{N-K}$ and $\eps>0$ such that $$\int_{B(0,R)}|u_n-v_n|^2\,dx\ge\eps.$$
	We can assume that $u_n\rightharpoonup u$, $v_n\rightharpoonup v$ in $X$ and $u\ne v$. Hence $\cJ'(u)=\cJ'(v)=0$ and consequently $$\liminf_n\|u_n-v_n\|\ge\|u-v\|\ge\kappa,$$ in contrast with the assumptions.
	
	Therefore it follows that $\lim_n\irn f(x,u_n)(u_n-v_n)\,dx=\lim_n\irn f(x,v_n)(u_n-v_n)\,dx=0$ and, finally,
	\begin{eqnarray*}
	\|u_n-v_n\|^2&=&\cJ'(u_n)(u_n-v_n)-\cJ'(v_n)(u_n-v_n)+\irn\bigl(f(x,u_n)-f(x,v_n)\bigr)(u_n-v_n)\,dx\\
	&=&o(1)+\irn f(x,u_n)(u_n-v_n)\,dx-\irn f(x,v_n)(u_n-v_n)\,dx\to0.\qedhere
	\end{eqnarray*}
\end{proof}

\begin{proof}[Proof of Theorem \ref{ExMul}]
	Note that $\cN$ contains all the nontrivial critical points of $\cJ$. Applying \cite[Theorem 3.3]{MeScSz} we obtain a Cerami sequence $(u_n)\subset X_{\cO}$ at level $c:=\inf_\cN\cJ>0$. Lemma \ref{lem:Mbeta}(a) implies that there exists $u\in X_{\cO}$ such that $u_n\rightharpoonup u$ up to a subsequence, thus $\cJ'(u)=0$.
	
	If by contradiction $\irn f(x,u_n)u_n\,dx\to0$, then similarly to the proof of Lemma \ref{lem:Mbeta}(b) we infer that $u_n\to0$, in contrast with $\cJ(u_n)\to c$. Hence, again similarly to the proof of Lemma \ref{lem:Mbeta}(b), $u\ne0$.
	
	Fatou's Lemma and (F4) imply
	\[\begin{split}
	c&=\lim_n\cJ(u_n)=\lim_n\cJ(u_n)-\frac12\cJ'(u_n)u_n=\lim_n\int_{\R^N}\frac12 f(x,u_n)u_n-F(x,u_n)\,dx\\
	&\ge\int_{\R^N}\frac12 f(x,u)u-F(x,u)\,dx=\cJ(u)-\frac12\cJ'(u)u=\cJ(u)\ge c
	\end{split}\]
	and we conclude $\cJ(u)=c$.
	
	Now assume $f$ is odd in $u$, which implies that $\cJ$ is even. The existence of infinitely many $\Z^{N-K}$-distinct critical points of $\cJ$ follows directly from \cite[Theorem 3.5(b)]{MeScSz}. As for the fact that the ground state solution is nonnegative, since $\cJ(u)=\cJ(|u|)$ and $\cJ'(u)u=\cJ'(|u|)|u|$ for $u\in X_{\cO}$, we can replace $(u_n)$ with $(|u_n|)$ and still we obtain a weak limit point, which is a nonnegative ground state solution.
\end{proof}
	
\begin{Lem}\label{positive}
	Suppose that $f$ does not depend on $y$ and satisfies (F5). Then there exists $w\in X_{\cO}$ such that $\irn F(z,w)\,dx>\frac{1}{2}\int_{\R^N}|\nabla_z w|^2 \, dx$.
\end{Lem}
\begin{proof}
Similarly as in \cite[page 325]{BerLions},
for any $R\ge3$ we define an even and continuous function $\phi_R\colon\R\to \R$ such that $\phi_R(t)=0$ for $|t|<1$ and for $|t|>R+1$, $\phi_R(t)=1$ for $2\leq |t|\leq R$ and $\phi_R$ is affine for $1\le|t|\le2$ and for $R\le|t|\le R+1$. Then let $w_R(x) := u_0 \phi_R(|y|) \phi_R(|z|)$. Observe that $w_R\in X_{\cO}$ and there are constants $C_1,C_2, C_3>0$ such that
\begin{eqnarray*}
\int_{\R^N} F(z,w_R)\,dx&\geq& C_1R^N\mathrm{essinf}_{z\in\R^{N-K}}F(z,u_0)-C_2R^{N-1}\sup_{R\leq |u|\leq R+1}\mathrm{esssup}_{z\in\R^{N-K}}|F(z,u)|\\
&&-C_3\sup_{1\leq |u|\leq 2}\mathrm{esssup}_{z\in\R^{N-K}}|F(z,u)|.
\end{eqnarray*}
Moreover,
$$\frac{1}{2}\int_{\R^N}|\nabla_z w_R|^2\leq C_4R^{N-1}$$
for some constant $C_4>0$.
Then for sufficiently large $R>0$ we conclude. 
\end{proof}
	
\begin{proof}[Proof of Theorem \ref{ExMul2}]
		First we prove that $\cJ$ has the mountain pass geometry \cite{AmbRab,Rabin}. Let $w\in X_{\cO}$ as in Lemma \ref{positive}. Due to
		(F2) and the embedding of $X_{\cO}$ into $L^{2^*}(\R^N)$,
		there exists $0<\rho<\|w\|$ such that $\inf\{\cJ(u):u\in X_{\cO}$ and $\|u\|=\rho\}>0$.
		Moreover, for every $\lambda>0$ we have
		\[
		\cJ\bigl(w(\lambda\cdot,\cdot)\bigr)=\frac{1}{2\lambda^{K-2}}\int_{\R^N}|\nabla_y w|^2+\frac{a}{r^2}|w|^2\,dx+\frac{1}{\lambda^K}\int_{\R^N}\frac12|\nabla_z w|^2-F(z,w)\,dx\to-\infty
		\]
		as $\lambda\to 0^+$. The existence of a Palais-Smale sequence $(u_n)\subset X$ for $\cJ|_{X_{\cO}}$ at the mountain pass level $c>0$ follows. Such a sequence is bounded because (F5) holds.
		
		Now, suppose by contradiction that \eqref{eq:LionsCond12} holds for every $R>0$. Fix $R>\sqrt{N-K}$ such that \eqref{eq:LionsCond12} holds with the supremum being taken over $\Z^{N-K}$. Since $F$ and $(z,u)\mapsto f(z,u)u$ satisfy \eqref{eq:LionsPhi} uniformly with respect to $z\in\R^{N-K}$, arguing as in Lemmas \ref{Lions2} and \ref{Lions} we obtain
		\[
		c=\lim_n\cJ(u_n)-\frac12\cJ'(u_n)u_n=\lim_n\int_{\R^N}\frac12 f(z,u_n)u_n-F(z,u_n)\,dx=0,
		\]
		which is a contradiction. It follows that there exist $R>\sqrt{N-K}$ and $\varepsilon>0$ such that
		\begin{equation}\label{eq:bddaway}
		\int_{B((0,z_n),R)}|u_n|^2\,dx\ge\varepsilon
		\end{equation}
		up to a subsequence, where $z_n\in\Z^{N-K}$ maximizes $z\mapsto\int_{B((0,z),R)}|u_n|^2\,dx$.
		Since $\cJ$ is invariant with respect to $\Z^{N-K}$ translations, up to replacing $u_n$ with $u_n(\cdot-z_n)$ we can suppose that $z_n=0$.
		Since $(u_n)$ is bounded, there exists $u\in X_{\cO}$ such that $u_n\rightharpoonup u$ in $X$, which in turn implies that $\cJ'(u)=0$ and that $u_n\to u$ in $L^2\bigl(B(0,R)\bigr)$ and a.e. in $\rn$; in particular, $u\ne 0$ because \eqref{eq:bddaway} holds.
	\end{proof}

\begin{proof}[Proof of Corollary \ref{CorExMul}]
The proof follows from Theorems \ref{ExMul}, \ref{ExMul2}, and \ref{ScalVec}.
\end{proof}

	\section*{Acknowledgements}
	\noindent The authors were partially supported by the National Science Centre, Poland (Grant No. 2017/26/E/ST1/00817). J. Mederski was also partially supported by the Alexander von Humboldt Foundation (Germany) and by the Deutsche Forschungsgemeinschaft (DFG, German Research Foundation) – Project-ID 258734477 – SFB 1173 during the stay at Karlsruhe Institute of Technology. Jacopo Schino is a member of GNAMPA (INdAM).

\end{document}